\documentclass[11pt]{amsart}              
\usepackage{latexsym,mathtools}    
\usepackage{epsfig,setspace}  
\usepackage{a4}
\usepackage{comment} 
\usepackage{mathrsfs}
\usepackage{amssymb}
\usepackage{amsthm}
\usepackage{amsbsy} 
\usepackage{upgreek}
\usepackage{relsize}
\usepackage{tikz-cd}
\usepackage{ytableau}

\usepackage{amsfonts,amssymb,latexsym,amscd,amsmath,euscript,enumerate,verbatim,calc}  
\allowdisplaybreaks
\usepackage{url}
\usepackage{hhline} 
\usepackage{pifont}
\usepackage{relsize}
\usepackage[toc,page]{appendix}
\usepackage[backend=biber,maxbibnames=999]{biblatex}  

\DeclareFieldFormat{doi}{\href{https://doi.org/#1}{\textsc{doi}}}

\DeclareFieldFormat{url}{\href{#1}{\textsc{url}}}
\renewbibmacro*{doi+eprint+url}{
  \printfield{doi}
  \newunit\newblock
  \iftoggle{bbx:eprint}{
    \usebibmacro{eprint}
  }{}
  \newunit\newblock
  \iffieldundef{doi}{
    \usebibmacro{url+urldate}}
  {}
}
\renewbibmacro{in:}{}
\AtEveryBibitem{
  \ifentrytype{book}
  {\clearfield{pages}}
}

\usepackage[colorlinks=true,linkcolor=blue,citecolor=magenta]{hyperref}

\theoremstyle{definition}

\newtheorem*{theorem*}{Theorem}
\newtheorem{theorem}{Theorem}[section]   
\newtheorem{lemma}[theorem]{Lemma}
\newtheorem{proposition}[theorem]{Proposition}
\newtheorem{question}[theorem]{Question}
\newtheorem{problem}[theorem]{Problem}
\newtheorem{conjecture}[theorem]{Conjecture} 
\newtheorem{corollary}[theorem]{Corollary} 

\newtheorem{definition}[theorem]{Definition}

\newtheorem{remark}[theorem]{Remark} 
\newtheorem{example}[theorem]{Example}

\def\F{{\mathbb F}}

\def\PP{{\mathbb P}}
\def\QQ{{\mathbb Q}}
\def\Fq{{\mathbb F}_q}
\def\LL{{\mathrm L}}

\def\MM{{\mathrm M}}
\def\mm{{\mathsf m}}
\def\NN{{\mathrm N}}
\def\he{\mathscr{H}}
\def\Irr{{\mathrm{Irr}}}
\def\Par{{\mathrm{Par}}}
\def\CC{{\mathcal{C}}}

\def\GL{{\mathrm{GL}}}

\def\ch{\mathrm{ch}}
\def\Ind{\mathrm{Ind}}
\newcommand{\resbrack}[2]{\genfrac{\{}{\}}{0pt}{}{#1}{#2}_{\mathrm{res}}}
\makeatletter
\@namedef{subjclassname@2020}{\textup{2020} Mathematics Subject Classification}
\makeatother

\title[Enumerating matrices with prescribed entries in an adjoint orbit]{Enumerating matrices with prescribed\\ entries in an adjoint orbit}    
\author{Samrith Ram} 
\address{Indraprastha Institute of Information Technology Delhi, New Delhi, India.}
\email{samrith@iiitd.ac.in}

\begin{document} 
\begin{abstract}
We study intersections of conjugacy classes of square matrices over a finite field with affine coordinate subspaces, or equivalently matrices in a fixed adjoint orbit with prescribed entries. Our main result treats the case of prescribed columns: for a partially defined linear map we give a Hall scalar product formula for the number of extensions to an endomorphism with prescribed similarity invariants. This formula is expressed in terms of skew modified Hall--Littlewood functions and $q$-Whittaker functions. As applications, we count monic matrix polynomials over $\Fq$ with prescribed Smith normal form and with prescribed determinant, and recover the Gerstenhaber--Reiner formula for the number of square matrices with a fixed characteristic polynomial. We also note that known point-count formulas for Hessenberg varieties imply related formulas for Hessenberg supports involving chromatic quasisymmetric functions, motivating polynomiality questions for more general supports and prescribed affine slices.
\end{abstract}
\maketitle
\setcounter{tocdepth}{2}   
\tableofcontents   

\section{Introduction}
Let $n$ be a positive integer and let $\Fq$ be the finite field with $q$ elements. Write $\MM_n(\Fq)$ for the algebra of $n\times n$ matrices over $\Fq$ and $\GL_n(\Fq)$ for its group of units. The conjugacy classes in $\MM_n(\Fq)$ are the orbits under the adjoint action of $\GL_n(\Fq)$ on $\MM_n(\Fq)$. Green~\cite{MR72878} parametrized these classes by functions from the set of monic irreducible polynomials over $\Fq$ to integer partitions; this parametrization is closely related to the theory of Hall polynomials and Hall--Littlewood symmetric functions as treated by Macdonald~\cite{MR1354144}.

The object of this paper is to enumerate matrices in a fixed adjoint orbit after constraining specified entries. Let $[n]=\{1,\ldots,n\}$ and let $S\subseteq [n]\times [n]$. We regard the entries in $S$ as free and prescribe the entries outside $S$.
\begin{definition}\label{def:prescribed-count}
Given a function $\varphi:([n]\times [n])\setminus S\to \Fq$ and a conjugacy class $\CC$ of $\MM_n(\Fq)$, let $N(\CC;S;\varphi)$ denote the number of matrices $A=(a_{ij})\in \CC$ such that $a_{ij}=\varphi(i,j)$ whenever $(i,j)\notin S$.
\end{definition}

Equivalently, the prescribed entries determine an affine coordinate subspace $\mathcal A_{S,\varphi}\subseteq \MM_n(\Fq)$, and $N(\CC;S;\varphi)=|\CC\cap \mathcal A_{S,\varphi}|.$ When $\varphi=0$, we write $V_S:=\mathcal A_{S,0}$ for the coordinate subspace consisting of matrices whose support is contained in $S$.

\begin{example}
Let $n=3$ and $q=5$. Take
\[
S=\{(1,1),(1,3),(2,2),(3,2),(3,3)\}
\]
and define $\varphi$ on the complement of $S$ by
\[
\varphi(1,2)=1,\qquad \varphi(2,1)=4,
\qquad \varphi(2,3)=3,
\qquad \varphi(3,1)=0.
\]
Then $N(\CC;S;\varphi)$ is the number of matrices in the conjugacy class $\CC$ of the form
\[
\begin{pmatrix}
* & 1 & * \\
4 & * & 3 \\
0 & * & *
\end{pmatrix}.
\]
\end{example}

We are primarily interested in the following problem.
\begin{problem}\label{prob:main}
Given $\CC,S$ and $\varphi$, determine $N(\CC;S;\varphi)$.
\end{problem}

In the affine-coordinate formulation above, Problem~\ref{prob:main} may be viewed as a finite-field point-counting problem for affine-linear slices of adjoint orbits.

We first recall several closely related enumeration problems that have appeared in the literature. In the specialization where $\varphi$ is identically zero, a support $S$ determines the coordinate subspace $V_S$, and Problem~\ref{prob:main} asks for the size of the intersection of a conjugacy class with a structured linear subspace of $\MM_n(\Fq)$. 
This zero-prescription point of view is closely connected with a substantial literature on finite-field matrix enumeration. Counting matrices of fixed rank with prescribed zero entries is a $q$-analogue of counting permutations with restricted positions, or equivalently rook placements~\cite{LewisLiuMoralesPanovaSamZhang2011}. For a coordinate subspace $V_S\subseteq \MM_n(\Fq)$, the number of matrices in $V_S$ of rank $r$ may be written as
\[
\sum_{\substack{\CC\subseteq \MM_n(\Fq)\\ \operatorname{rank}(\CC)=r}} |\CC\cap V_S|,
\]
where the sum is over adjoint orbits of rank $r$. Thus orbit-support intersections refine the usual restricted-rank enumeration problem. Haglund~\cite{Haglund1998QRook} proved polynomiality and positivity results for Young-diagram supports using $q$-rook polynomials; related questions for restricted entries were studied by Lewis, Liu, Morales, Panova, Sam and Zhang~\cite{LewisLiuMoralesPanovaSamZhang2011}. Klein, Lewis and Morales~\cite{KleinLewisMorales2013} extended these ideas to complements of skew Young diagrams and Rothe diagrams, and emphasized that for arbitrary supports one should not expect polynomiality in $q$: Stembridge~\cite{Stembridge1998Kontsevich} had already found examples, related to a conjecture of Kontsevich, where the corresponding point counts are not polynomials in $q$. Our results suggest analogous questions in which rank is replaced by the datum of conjugacy-class type (Definition~\ref{def:class-type}).

Another important family of examples comes from triangular and Hessenberg supports (Section~\ref{subsec:hessenberg-spaces}). The distribution of Jordan forms among upper triangular nilpotent matrices has been studied from several viewpoints, including the work of Kirillov and Melnikov~\cite{MR1601186}, Ekhad and Zeilberger~\cite{MR1364064}, Yip~\cite{Yip2018RookJordan}, and Fuchs and Kirillov~\cite{FuchsKirillov2022Jordan}. Hessenberg varieties (Definition~\ref{def:hessenberg-variety}), introduced by De Mari, Procesi and Shayman~\cite{MR1043857} and developed in works of Tymoczko~\cite{MR2275912} and Brosnan and Chow~\cite{MR3783432}, connect these enumeration problems with chromatic quasisymmetric functions of Shareshian and Wachs~\cite{MR3488041} (defined in Section~\ref{subsec:hessenberg-spaces}). This connection is made explicit in recent point-counting formulas for Hessenberg varieties due to Abreu, Nigro and Ram~\cite{abreu2024}, Gagnon~\cite{MR4659580}, and Bardestani, Mallahi-Karai, Ram and Salmasian~\cite{BardestaniMallahiKaraiRamSalmasian2026}.

We now describe the main result of this paper. It treats Problem~\ref{prob:main} in the case where the prescribed entries consist of specified columns. If a subset $C\subseteq [n]$ of columns is prescribed, then the support of the free entries is $S=[n]\times ([n]\setminus C)$, and the prescribed columns define a linear map on the coordinate subspace spanned by $\{e_j:j\in C\}$. Thus this part of Problem~\ref{prob:main} becomes one of counting extensions of a partially defined linear transformation to an endomorphism with prescribed similarity invariants. 
For an ordinary operator over $\Fq$, rational canonical form says that its similarity class is encoded by a function from $\Irr_q$ to $\Par$, where $\Irr_q$ is the set of monic irreducible polynomials over $\Fq$ and $\Par$ is the set of partitions. The classification for linear maps defined only on a subspace is less standard, so we spell out the notation used below; see Section~\ref{sec:preliminaries}, especially Theorem~\ref{thm:similarity} and the notation following Equation~\eqref{eq:sct}, for details. If $T$ is such a partial linear map, then $T$ has a maximal invariant subspace, and its operator part is simply the operator obtained by restricting $T$ to this subspace. The similarity invariants of $T$ consist of an integer partition $\mu$, called the sequence of defect dimensions, and the ordinary operator invariant $\MM$ of its operator part. Thus $\MM$ is a function from $\Irr_q$ to $\Par$. 
Since $\MM(f)$ is empty for all but finitely many $f$, we write $\MM=\{(f_1,\mu^1),\ldots,(f_r,\mu^r)\}$ when $\MM(f_i)=\mu^i$ and $\MM(g)$ is empty for every other irreducible polynomial $g$.

Let $T$ be a linear map defined on a subspace of $\Fq^n$ with similarity invariants $(\mu,\MM)$, and let $\eta((\mu,\MM),\LL)$ denote the number of extensions of $T$ to an operator on $\Fq^n$ with similarity invariants $\LL$. The following theorem is our main result.
\begin{theorem*}[Theorem~\ref{thm:eta} below]
We have
\[
\eta((\mu,\MM),\LL)=
\frac{1}{{n \brack \mu_1}_q}\frac{|\CC(\LL)|}{\pi(\mu,\MM)}
\langle A_{\LL\MM}(x;q),B_\mu(x;q)\rangle.
\]
\end{theorem*}

Here $\CC(\LL)$ is the conjugacy class indexed by $\LL$, and ${n \brack \mu_1}_q$ is a $q$-binomial coefficient. The factor $\pi(\mu,\MM)$ is an explicit normalization factor counting maps with similarity invariants $(\mu,\MM)$ (Proposition~\ref{prop:partialsize}). The symmetric functions in the Hall scalar product are as follows. If $\LL=\{(f_1,\lambda^1),\ldots,(f_r,\lambda^r)\}$ and $\MM=\{(f_1,\mu^1),\ldots,(f_r,\mu^r)\}$ with $d_i=\deg f_i$, allowing empty partitions as needed, then
\[
A_{\LL\MM}(x;t)=\prod_{i=1}^r p_{d_i}\circ \widetilde{H}_{\lambda^i/\mu^i}(x;t),
\]
where $\widetilde{H}_{\lambda/\mu}(x;t)$ is a skew modified Hall--Littlewood polynomial (Definition~\ref{def:skew-modified-hl}) and $\circ$ denotes plethystic substitution. The second factor is
\[
B_\mu(x;t)=(-1)^{\sum_{j\geq 2}\mu_j}t^{\sum_{j\geq 2}{\mu_j\choose 2}}\widetilde{W}_\mu(x;t),
\]
where $\widetilde{W}_\mu$ is the dual $q$-Whittaker function, defined in terms of the Hall--Littlewood $P$-function by $\widetilde{W}_\mu=\omega P_{\mu'}$. Here $\omega$ denotes the standard involution on symmetric functions. In matrix language, Theorem~\ref{thm:eta} gives the number of matrices in a conjugacy class whose first $k$ columns are prescribed (see Theorem~\ref{thm:eta-matrix} below); by replacing the coordinate subspace, the same formula applies to any fixed set of $k$ prescribed columns.

Several applications of Theorem \ref{thm:eta} stated above are developed in Section~\ref{sec:matrixcounting}. First, we count monic matrix polynomials of fixed degree and prescribed Smith normal form (Section~\ref{subsec:matrix-polynomials}). 
Second, by summing over conjugacy classes with a fixed characteristic polynomial, we obtain a symmetric-function formula for the number of extensions of a partial map with a prescribed characteristic polynomial. In the case where no columns are prescribed, this gives a new proof of the Gerstenhaber--Reiner formula for the number of matrices over $\Fq$ with a fixed characteristic polynomial.

The paper is organized as follows. Section~\ref{sec:preliminaries} recalls the similarity classification of linear maps defined on subspaces, formulates the basic extension-counting problem, and introduces the symmetric-function notation used throughout the paper. It culminates in the Hall-scalar-product formula of Theorem~\ref{thm:eta}, together with its interpretation as a prescribed-column matrix count in Theorem~\ref{thm:eta-matrix}. In Section~\ref{sec:representation-theory} we discuss connections to certain induced representations, and irreducible unipotent characters of $\GL_n(\Fq)$. Section~\ref{sec:matrixcounting} develops applications to matrix enumeration, including  monic matrix polynomials with prescribed Smith normal form, and matrices with prescribed columns and characteristic polynomial. The last application recovers the Gerstenhaber--Reiner formula for the number of matrices with a given characteristic polynomial as a special case. Section~\ref{sec:supports-conjugacy} turns from prescribed columns to more general coordinate supports. There we discuss connections with chromatic quasisymmetric functions and formulate polynomiality questions for $N(\CC;S;\varphi)$. Finally, Section~\ref{sec:conclusion} concludes with natural questions arising from this work.
\section{Preliminaries}\label{sec:preliminaries}
In this section we recall some enumerative results of Arora and Ram~\cite{MR4264825} on conjugacy classes of linear maps that are defined on subspaces of a finite vector space. Let $\Irr_q$ denote the set of all monic irreducible polynomials over $\Fq$ in the variable $x$. A partition of a nonnegative integer $n$ is a weakly decreasing sequence $\lambda=(\lambda_1,\lambda_2,\ldots)$ of nonnegative integers satisfying $|\lambda|:=\sum_{i\geq 1}\lambda_i=n$. As is customary, one omits trailing zeros while writing partitions; for example, $(4,2,2,1,0,0,\ldots)=(4,2,2,1)$. Denote by $\Par$ the set of all partitions of nonnegative integers.

Throughout this paper, $V$ denotes a vector space of dimension $n$ over $\Fq$ and $W$ denotes a subspace of $V$. Denote by $L(W,V)$ the space of all $\Fq$-linear transformations from $W$ to $V$.
\begin{definition}\label{def:subspace-similarity}
Two linear transformations $T$ and $\hat{T}$ defined on subspaces $W$ and $\hat{W}$ of $V$ are similar if there exists a linear isomorphism $S:V\to V$ such that $S(W)=\hat{W}$ and $S\circ T=\hat{T}\circ S_W$, where $S_W$ denotes the restriction of $S$ to $W.$ In other words, the following diagram commutes: 
$$
\begin{tikzcd}
W \arrow{r}{T} \arrow[swap]{d}{S_W} & V \arrow{d}{S}[swap]{\simeq} \\
\hat{W} \arrow{r}{\hat{T}} & V
\end{tikzcd}
$$
\end{definition}
Note that the above definition coincides with the usual notion of similarity in the case $W=V$.

We now give a complete set of invariants which characterize the similarity class of a map defined on a subspace (see Gohberg, Kaashoek and van~Schagen~\cite[Sec.~III.1]{Gohbergetal1995}). Given $T\in L(W,V)$, define a sequence of subspaces $W_i(i\geq 0)$ by setting $W_0=V, W_1=W$ and, for each $i\geq 1$, $W_{i+1}=W_i\cap T^{-1}(W_i)$, where $T^{-1}$ denotes the inverse image under $T.$ We therefore have a descending chain $W_0\supseteq W_1\supseteq\cdots$ of subspaces which eventually stabilizes. For each $i\geq 0$, write $d_i=\dim W_i$ and let $\ell=\min\{i:W_i=W_{i+1}\}$. It then follows that $W_\ell$ is the maximal $T$-invariant subspace. The linear operator $T'$ obtained by restricting $T$ to $W_\ell$ is called the operator part of $T$. Let $\mu_j=d_{j-1}-d_j$ for each $j\geq 1$ and write $\mu=\mu(T):=(\mu_1,\mu_2,\ldots)$. It is known that $\mu$ is an integer partition, called the sequence of \emph{defect dimensions} of $T$ (see Gohberg, Kaashoek and van~Schagen~\cite[pg.~52]{Gohbergetal1995}). The similarity class of $T$ is then uniquely determined by the integer partition $\mu$ and the similarity invariants of the operator $T'$ as the following theorem \cite[Thm.~III.1.4]{Gohbergetal1995} shows.  

\begin{theorem}\label{thm:similarity}
  Two linear maps $T$ and $\hat{T}$ defined on subspaces $W$ and $\hat{W}$ of $V$ are similar if and only if $\mu(T)=\mu(\hat{T})$ and their operator parts $T'$ and $\hat{T}'$ are similar.
\end{theorem}
\begin{remark}
 The integers $d_j$ above are given by $d_j=\dim(W\cap T^{-1}W\cap \cdots \cap T^{-(j-1)}W)$ for each $j\geq 1$.
\end{remark}

Similarity of linear operators on $\Fq^n$ is well-studied and can naturally be identified with similarity of square matrices.  It is known (Stanley~\cite[Sec.~1.10.1]{Stanley2012}) that conjugacy classes of $n\times n$ matrices over $\Fq$ are indexed by functions $\LL:\Irr_q\to \Par$ which satisfy  
\begin{equation}\label{eq:sct}
 \lVert \LL \rVert:= \sum_{f \in \Irr_q} |\LL(f)| \deg f=n.
\end{equation}

The sum above is finite as $\LL$ necessarily maps all but finitely many elements of $\Irr_q$ to the empty partition. The conjugacy class indexed by $\LL$ is denoted $\CC(\LL)$. 

\begin{remark}
The conjugacy classes can be described by the structure theorem for finitely generated modules over a principal ideal domain. Each linear operator $T$ on $\Fq^n$ defines an $\Fq[x]$-module structure on $\Fq^n$ where the action of $x$ on a vector is that of $T$. This module is isomorphic to a direct sum
\begin{align*}
  \bigoplus_{i=1}^r\bigoplus_{j\geq 1}\frac{\Fq[x]}{(f_i^{\lambda_{i,j}})},
\end{align*}
where the polynomials $f_i\in \Irr_q$ are distinct and the sequence $\lambda^i=(\lambda_{i,1}, \lambda_{i,2}, \ldots)$ is an integer partition for each $1\leq i\leq r.$ The function indexing the conjugacy class of $T$ is simply $\LL=\{(f_i,\lambda^i)\}_{1\leq i\leq r}$.
\end{remark}

If $\{f_1,\ldots,f_r\}$ is a finite subset of $\Irr_q$ such that $\LL(g)$ is the empty partition for every polynomial $g\in\Irr_q\setminus\{f_1,\ldots,f_r\}$ we also write $\LL=\{(f_1,\lambda^1),\ldots,(f_r,\lambda^r)\}$ where $\lambda^i=\LL(f_i)$ for $1\leq i\leq r$. It will occasionally be convenient to allow some of the partitions $\lambda^i$ above to be empty. By a slight abuse of notation, we write $\emptyset$ for the function indexing the conjugacy class of the unique operator on a zero dimensional space  (this seems natural since it corresponds to the function which assigns the empty partition to each irreducible polynomial).

Philip Hall's explicit formula, as stated by Stanley~\cite[Eqn.~1.107]{Stanley2012}, for the size of the conjugacy class indexed by $\LL=\{(f_1,\lambda^1),\ldots,(f_r,\lambda^r)\}$ is
\begin{align*}
  |\CC(\LL)|=\frac{|\GL_n(\Fq)|}{\prod_{i=1}^r c_q(d_i,\lambda^i)},
\end{align*}
where $d_i=\deg f_i$ for $1\leq i\leq r$ and
\begin{equation}\label{eq:centralizer}
  c_q(d,\lambda):=q^{d\sum_{i\geq 1}(\lambda'_i)^2}\prod_{i\geq 1}(q^{-d};q^{-d})_{m_i(\lambda)}.
\end{equation}
Here $\lambda'$ denotes the partition conjugate to $\lambda$ and $(q;q)_k:=\prod_{j=1}^k (1-q^j)$ is a $q$-Pochhammer symbol while $m_i(\lambda)$ denotes the multiplicity of the integer $i$ as a part of $\lambda$. The quantity $c(\LL):=\prod_{i=1}^r c_q(d_i,\lambda^i)$ represents the cardinality of the $\GL_n(\Fq)$-centralizer corresponding to the class $\LL$.

In view of Theorem \ref{thm:similarity} each similarity class of linear maps in $L(W,V)$ is indexed by a pair $(\lambda,\LL)$, where $\lambda$ is an integer partition with $\lambda_1=\dim V-\dim W$ and $\LL$ indexes a conjugacy class of operators satisfying $|\lambda|+\lVert \LL \rVert=n$. When $\lVert \LL \rVert=\dim V,$ the conjugacy class $\CC(\LL)$ has similarity invariants $(\emptyset,\LL)$, where $\emptyset$ denotes the empty partition of 0.

\begin{example}
  Let $V=\Fq^6$ and let $e_i(1\leq i\leq 6)$ denote the standard basis vectors of $V.$ Consider the linear map $T$ defined on the subspace $W={\rm span}\{e_1,e_2,e_3,e_4\}$  by
  \begin{align*}
    e_1\xrightarrow{T} e_2 \xrightarrow{T}e_5 \qquad  e_3\xrightarrow{T} e_4 \xrightarrow{T} 0.
  \end{align*}
  Then $\dim V=6,\dim W=4$. We have
  \begin{align*}
    \dim(W\cap T^{-1}W)&=\dim{\rm span}\{e_1,e_3,e_4\}=3,\\
     \dim(W\cap T^{-1}W\cap T^{-2}W)&=\dim{\rm span}\{e_3,e_4\}=2.
  \end{align*}
 Since ${\rm span}\{e_3,e_4\}$ is $T$-invariant, the defect dimensions of $T$ are given by $\lambda=(6-4,4-3,3-2)=(2,1,1)$. The operator part of $T$ is nonzero and  nilpotent; therefore it has similarity invariants $\LL=\{(x,(2))\}$.
\end{example}

\begin{definition}\label{def:restriction-count}
For a map $T\in L(W,V)$ with similarity invariants $(\lambda,\LL)$, let
$$
\resbrack{\lambda,\LL}{\mu,\MM}
$$
denote the number of subspaces $U \subseteq W$ such that the restriction $T_U:U \to V$ has similarity invariants $(\mu,\MM)$. 
\end{definition}
If either $\lambda$ or $\mu$ is the empty partition we omit it from the notation. For example, we write $\resbrack{\LL}{\mu,\MM}$ for $\resbrack{\emptyset,\LL}{\mu,\MM}$. We are interested in the following problem.
\begin{problem}\label{prob:coeff}
Given $(\lambda, \LL)$ and $(\mu,\MM)$, find an explicit formula for
  $$
\resbrack{\lambda,\LL}{\mu,\MM}.
$$
\end{problem}
The restriction relation also defines a partial order on similarity-invariant pairs.
\begin{problem}\label{prob:restriction-poset}
Declare $(\mu,\MM) \preceq (\lambda,\LL)$ if $(\mu,{\mathrm M})$ occurs as the similarity invariants of the restriction of some map with similarity invariants $(\lambda,{\mathrm L})$. Describe the combinatorial properties of this poset.
\end{problem}
Finding an explicit answer to Problem \ref{prob:coeff} even in the case $\LL=\MM=\emptyset$ appears to be open. Theorem \ref{thm:coeffcomputation} gives an explicit formula involving symmetric functions to Problem \ref{prob:coeff} in the case where $\lambda$ is the empty partition. 
\begin{example}\label{eg:identity}
The identity operator on $\Fq^n$ lies in the conjugacy class corresponding to $\LL=\{(x-1,(1^n))\}$. If $\MM=\{(x-1,(1^k))\}$, then we have
\begin{align*}
  \resbrack{\LL}{(n-k),\MM}={n \brack k}_q,
\end{align*}
a $q$-binomial coefficient, defined as the number of $k$-dimensional subspaces of $\Fq^n$.
\end{example}
\begin{definition}\label{def:simple-map}
  A linear map $T$ defined on a subspace of $V$ is said to be \emph{simple} if, for each $T$-invariant subspace $U$, either $U=\{0\}$ or $U=V$.
\end{definition}

\begin{example}
  The linear map $T$ defined on the subspace $W={\rm span}\{e_1,e_2,e_3\}$ of $\Fq^6$  by
  \begin{align*}
    e_1\xrightarrow{T} e_2 \xrightarrow{T}e_4 \\ e_3\xrightarrow{T} e_5 
  \end{align*}
is simple with defect dimensions $\lambda=(3,2,1)$ and has similarity invariants $(\lambda,\emptyset)$.
\end{example} 

\begin{remark}
 If $W$ is a proper subspace of $V$, then a map $T\in L(W,V)$ is simple precisely when $T$ has similarity invariants $(\lambda,\emptyset)$ for some partition $\lambda$ of $\dim V.$ On the other hand, an operator on $V$ is simple precisely when it has an irreducible characteristic polynomial.
\end{remark}

\begin{lemma}
 Given a linear map $T:W\to \Fq^n$, the number of extensions of $T$ to all of $\Fq^n$ with similarity invariants $\LL$ depends solely on the similarity invariants of $T$. 
\end{lemma}
\begin{proof}
  Suppose $T_1:W_1\to \Fq^n$ and $T_2:W_2\to \Fq^n$ have the same similarity invariants. By Theorem~\ref{thm:similarity}, the maps $T_1$ and $T_2$ are similar. Hence, by Definition~\ref{def:subspace-similarity}, there exists $g\in \GL_n(\Fq)$ such that $g(W_1)=W_2$ and
  \[
    T_2g=gT_1 \quad \text{on } W_1.
  \]
  If $A$ is an extension of $T_1$ to an operator on $\Fq^n$, then $gAg^{-1}$ is an extension of $T_2$. Conversely, conjugating by $g^{-1}$ sends extensions of $T_2$ to extensions of $T_1$. These two operations are inverse bijections, and conjugation preserves similarity invariants. Hence the number of extensions with prescribed similarity invariants $\LL$ is the same for $T_1$ and $T_2$.
\end{proof}

The following result of Arora and Ram~\cite[Cor.~4.7]{MR4264825} gives a formula for the number of maps which have specified similarity invariants.
\begin{proposition}\label{prop:partialsize}
For each subspace $W\subseteq V$, the number of maps $T\in L(W,V)$ which have similarity invariants $(\lambda,\LL)$ is given by 
  \begin{align}\label{eq:pidef}
    \pi(\lambda,\LL)=q^{d(k-d)}{k \brack d}_q |\CC(\LL)|\; \pi(\lambda,\emptyset),
  \end{align}
 where $d=\lVert \LL \rVert$ and  $k=\dim W$. Here
  \begin{align*}
\pi(\lambda,\emptyset)=\gamma_q(k-d)\; q^{\sum_{j\geq 2}\lambda_j^2}\; \prod_{i\geq 1}{\lambda_i \brack \lambda_{i+1}}_q,
  \end{align*}
  where $\gamma_q(r):=|\GL_r(\Fq)|=\prod_{i=0}^{r-1}(q^r-q^i)$.
\end{proposition}
\begin{remark}
 If $\lambda$ has a single part equal to $\dim V-\dim W$, then $W$ is $T$-invariant and $k=d$. In this case Proposition \ref{prop:partialsize} reduces to $\pi(\lambda,\LL)=|\CC(\LL)|$.  
\end{remark}

\begin{definition}\label{def:extension-count}
Let $T:W\to \Fq^n$ be a linear map with similarity invariants $(\mu,\MM)$. Define $\eta((\mu,\MM),\LL)$ to be the number of extensions of $T$ to all of  $\Fq^n$ which have similarity invariants $\LL$.
\end{definition}
The quantity  $\eta((\mu,\MM),\LL)$ will prove useful in counting various classes of matrices over $\Fq$.
\begin{proposition}\label{prop:eta}
  If $\lVert \LL \rVert=n$, then
  \begin{align*}
    \eta((\mu,\MM),\LL)=\frac{1}{{n \brack \mu_1}_q}\frac{|\CC(\LL)|}{\pi(\mu,\MM)}\,\resbrack{\LL}{\mu,\MM}.
  \end{align*}
\end{proposition}
\begin{proof}
  Count pairs $(W,T)$ where $W$ is a $k$-dimensional subspace of $\Fq^n$ and $T$ is a linear operator on $\Fq^n$ which lies in $\CC(\LL)$ with the property that the restriction $T_W$ of $T$ to $W$ has similarity invariants $(\mu,\MM)$ (with $\mu_1=n-k$). For a fixed $T\in \CC(\LL)$, the number of choices for $W$ is clearly $\resbrack{\LL}{\mu,\MM}$. Therefore the total number of such pairs is given by
  \begin{align*}
    \sum_{T\in \CC(\LL)} \resbrack{\LL}{\mu,\MM}=|\CC(\LL)|\,\resbrack{\LL}{\mu,\MM}.
  \end{align*}
  On the other hand, if $W$ is fixed, first choose a linear map $T':W\to \Fq^n$ with similarity invariants  $(\mu,\MM)$ and then extend $T'$ to all of $\Fq^n$ in $\eta((\mu,\MM),\LL)$ ways. Since there are ${n \brack k}_q$ subspaces $W$ of dimension $k$, it follows that the desired number of pairs $(W,T)$ is given by
  \begin{align*}
    {n \brack k}_q \pi(\mu,\MM)\; \eta((\mu,\MM),\LL),
  \end{align*}
 and the proposition follows.
\end{proof}
Note that for $\mu=(n)$ and $\MM=\emptyset$ the formula of Proposition \ref{prop:eta} reduces to $\eta((\mu,\MM),\LL)=|\CC(\LL)|$ as expected. In view of Proposition \ref{prop:eta}, it is desirable to have an explicit formula for $\resbrack{\LL}{\mu,\MM}$. Our immediate goal is to derive a formula involving symmetric functions for the coefficient $\resbrack{\LL}{\mu,\emptyset}$ when $\lVert \LL \rVert=|\mu|.$
\subsection{Symmetric Functions}\label{subsec:symmetric-functions}
We begin with a brief overview of symmetric functions; the reader is referred to Macdonald \cite{MR1354144} for the details. Denote by $\QQ(t)$ the field of rational functions in an indeterminate $t$. Write $\Lambda_t$ for the algebra of symmetric functions in infinitely many variables $x=(x_1,x_2,\ldots)$ with coefficients in $\QQ(t)$. 
The algebra $\Lambda_t$ admits several natural bases indexed by integer partitions such as the monomial ($m_\lambda$), complete homogeneous ($h_\lambda$), elementary ($e_\lambda$), power sum ($p_\lambda$) and Schur ($s_\lambda$) symmetric functions. There exists an involutory automorphism $\omega$ of $\Lambda_t$ satisfying $\omega e_\lambda=h_\lambda$. The algebra $\Lambda_t$ is also equipped with the Hall scalar product $\langle  \cdot,\cdot\rangle$, a symmetric bilinear form characterized by the property $\langle  m_\lambda,h_\mu \rangle =\delta_{\lambda\mu}$, where $\delta$ denotes the Kronecker delta function. Two bases $b_\lambda$ and $\tilde{b}_\lambda$ indexed by integer partitions are dual with respect to the Hall scalar product if $\langle b_\lambda,\tilde{b}_\mu \rangle=\delta_{\lambda\mu}$.

The Hall--Littlewood $P_\lambda(x;t)$ and $q$-Whittaker functions $W_\lambda(x;t)$ are two further parametric bases for $\Lambda_t$. Both these bases occur as specializations of the two-parameter Macdonald polynomials $P_\lambda(x;q,t)$; the Hall--Littlewood functions are the $q=0$ specialization while the $q$-Whittaker functions are obtained by setting $t=0$ and $q=t$. The basis dual to $P_\lambda(x;t)$ is $H_\lambda(x;t)$, the transformed Hall--Littlewood polynomials, given by
  \begin{align*}
      H_\lambda(x;t)=\left(\prod_{i\geq 1}(t;t)_{m_i(\lambda)}\right) P_\lambda\left[\frac{X}{1-t}\right],
    \end{align*}
    where $f[g]$ (also written $f\circ g$) indicates plethystic substitution of $g$ into $f$ (see Macdonald \cite[p.\ 135]{MR1354144} for the definition of plethysm). The function $H_\lambda(x;t)$ also satisfies the relation $W_\lambda(x;t)=\omega H_{\lambda'}(x;t)$. The modified Hall--Littlewood polynomials $\widetilde{H}_\lambda$ are defined by
    \begin{align*}
      \widetilde{H}_\lambda(x;t)=t^{n(\lambda)}H_\lambda(x;t^{-1}),
    \end{align*}
    where $n(\lambda)=\sum_{i\geq 1}(i-1)\lambda_i$. The modified Hall--Littlewood polynomials are more combinatorial in nature, being the Frobenius image of a graded vector space. The basis dual to $W_\lambda(x;t)$ with respect to the Hall scalar product is denoted $\widetilde{W}_\lambda(x;t)$ and satisfies $\widetilde{W}_\lambda(x;t)=\omega P_{\lambda'}(x;t)$.

The following notion of class type goes back to Green~\cite{MR72878}. 
\begin{definition}\label{def:class-type}
If $\LL=\{(f_1,\lambda^1),\ldots,(f_r,\lambda^r)\}$, then the \emph{class type} of the conjugacy class $\CC(\LL)$ is the multiset $\tau=\{(d_1,\lambda^1),\ldots,(d_r,\lambda^r)\}$, where $d_i=\deg f_i$. The size of $\tau$ is
\[
\lVert \tau\rVert:=\sum_{i=1}^r d_i|\lambda^i|=\lVert \LL\rVert.
\]
For a class type $\tau=\{(d_1,\lambda^1),\ldots,(d_r,\lambda^r)\}$, define
\[
F_\tau(x;t):=\prod_{i=1}^r p_{d_i}[\widetilde{H}_{\lambda^i}(x;t)],
\]
where the square brackets denote plethystic substitution. After specializing $t=q$, one has $p_d[\widetilde{H}_\lambda(x;t)]|_{t=q}=p_d[\widetilde{H}_\lambda(x;q^d)]$. If $\tau$ is the class type of $\LL$, we also write $F_\LL(x;t):=F_\tau(x;t)$.
\end{definition}

\begin{definition}\label{def:invflag}
Given a linear operator $T$ on $\Fq^n$ and a partition $\nu=(\nu_1,\nu_2,\ldots,\nu_\ell)$ of $n$, let $E_\nu(T)$ denote the number of flags
$$
(0)=V_0\subset V_1\subset \cdots \subset V_\ell=\Fq^n
$$
of $T$-invariant subspaces $V_i$ such that $\dim(V_i/V_{i-1})=\nu_i$ for $1\leq i\leq \ell$. The invariant flag generating function of $T$ is the symmetric function given in the monomial basis by
\begin{align*}
F_T(x)=\sum_{\nu\vdash n}E_\nu(T)m_\nu.
\end{align*}
\end{definition}

Suppose $T$ has similarity invariants $\LL=\{(f_1,\lambda^1),\ldots,(f_r,\lambda^r)\}$, and let $\tau$ be the class type of $\LL$. The invariant flag generating function depends only on $\tau$. More precisely, by~\cite[Prop.~2.11]{ram2023subspace},
\begin{align}\label{eq:class-type-flag}
F_T(x)=F_\tau(x;q)=F_\LL(x;q).
\end{align}
Thus $F_\tau(x;t)$, equivalently $F_\LL(x;t)$, is the universal class-type version of the invariant flag generating function, and the concrete invariant flag generating function of a matrix over $\Fq$ is obtained by specializing $t=q$.

\begin{example}\label{eg:invflag-examples}
We have the following specializations of Equation~\eqref{eq:class-type-flag}.
\begin{enumerate}
\item If $T$ is primary of type $(f,\lambda)$, where $d=\deg f$, then
\begin{align*}
F_T(x)=p_d\circ\widetilde{H}_{\lambda}(x;q^d).
\end{align*}
In particular, if $T$ is nilpotent with Jordan form partition $\lambda$, then $f=x$, $d=1$, and $F_T(x)=\widetilde{H}_{\lambda}(x;q).$
Thus a regular nilpotent operator has $F_T(x)=\widetilde{H}_{(n)}(x;q)=h_n$, while the zero operator has
\begin{align*}
F_T(x)=\widetilde{H}_{(1^n)}(x;q)
=\sum_{\nu\vdash n}\frac{[n]_q!}{[\nu_1]_q!\cdots [\nu_{\ell(\nu)}]_q!}m_\nu,
\end{align*}
where $[a]_q!=\prod_{j=1}^a(1+q+\cdots+q^{j-1})$.
\item If $T$ is semisimple and the distinct irreducible factors $f_i$ of its characteristic polynomial have degrees $d_i$ and multiplicities $m_i$, then
\begin{align*}
F_T(x)=\prod_i p_{d_i}[\widetilde{H}_{(1^{m_i})}(x;q^{d_i})].
\end{align*}
For a split semisimple operator with eigenspace dimensions $m_1,\ldots,m_r$, this reduces to
\begin{align*}
F_T(x)=\prod_{i=1}^r \widetilde{H}_{(1^{m_i})}(x;q).
\end{align*}
The scalar case is the special case $r=1$.
\item An operator $T$ is \emph{regular} if its minimal and characteristic polynomials coincide. If $T$ is regular semisimple, write its squarefree characteristic polynomial as a product of distinct irreducibles $f_1\cdots f_r$ and set $d_i=\deg f_i$. Then
\begin{align*}
F_T(x)=p_{d_1}p_{d_2}\cdots p_{d_r}.
\end{align*}
Hence a regular semisimple split operator has $F_T(x)=p_1^n=h_1^n$, while a regular semisimple operator with irreducible characteristic polynomial has $F_T(x)=p_n=m_n$.
\item More generally, if $T$ is regular with similarity invariants $\{(f_i,(m_i))\}_{i=1}^r$, then
\begin{align*}
F_T(x)=\prod_{i=1}^r p_{d_i}[h_{m_i}],\qquad d_i=\deg f_i.
\end{align*}
In the regular split case, where $f_i=x-a_i$ for distinct $a_i\in\Fq$, this becomes $F_T(x)=h_{m_1}\cdots h_{m_r}$.
\end{enumerate}
\end{example}

\begin{definition}\label{def:subspace-profile}
  Given a linear operator $T$ on $V$, we say that a subspace $W\subseteq V$ has $T$-profile $\mu=(\mu_1,\mu_2,\ldots)$ if
  \begin{align*}
    \dim (W+TW+\cdots +T^{j-1}W)=\mu_1+\mu_2+\cdots+\mu_j \mbox{ for }j\geq 1.
  \end{align*}
\end{definition}
Let $\sigma(\mu,T)$ denote the number of subspaces with $T$-profile $\mu$. In 1992, Bender, Coley, Robbins and Rumsey \cite[pg.~2]{MR1141317} proposed the problem of finding a formula for $\sigma(\mu,T)$. Specific instances of this problem for various types of operators have been studied~\cite{MR2831705,MR2961399,MR3093853,MR4263652,MR4349887,MR4555237,MR4682040,MR4797454,ram2023diagonal} and the following general solution was obtained in~\cite[Thm.~1.4]{ram2023subspace}.

\begin{theorem}\label{thm:cthulhu}
For each  integer partition $\mu$ and linear operator $T$ over $\Fq$ with similarity invariants $\LL=\{(f_1,\lambda^1),\ldots,(f_r,\lambda^r)\},$ 
  \begin{align*}
    \sigma(\mu,T)=(-1)^{\sum_{j\geq 2}\mu_j} q^{\sum_{j\geq 2}{\mu_j \choose 2}}\langle  F_\LL(x;q),\widetilde{W}_{\mu}(x;q)\, h_{n-|\mu|} \rangle.
  \end{align*}
Here $F_\LL(x;t)$ is the universal invariant flag generating function from Equation~\eqref{eq:class-type-flag}.
\end{theorem}
Theorem \ref{thm:cthulhu} will be used to derive an explicit expression for the coefficients $\resbrack{\LL}{\mu,\emptyset}$. The following result of Ram~\cite[Cor.~2.4]{ram2024} relates subspace profiles to defect dimensions.

\begin{lemma}\label{lem:restrictionformula} 
  For each linear operator $T$ on $V$, the number of subspaces $W$ such that the restriction of $T$ to $W$ has defect dimensions $\mu$ is given by $\sigma(\mu,T)$. In particular, given a conjugacy class $\CC(\LL)$ over $\Fq$ and a partition $\mu$ with $|\mu|=\lVert \LL \rVert$, we have
  \begin{align*}
    \resbrack{\LL}{\mu,\emptyset}=(-1)^{\sum_{j\geq 2}\mu_j} q^{\sum_{j\geq 2}{\mu_j \choose 2}}\langle  F_\LL(x;q),\widetilde{W}_{\mu}(x;q)\rangle.
  \end{align*}
\end{lemma}

Our next objective is to derive a formula for $\resbrack{\LL}{\mu,\MM}$ for arbitrary values of the parameters. We begin by recalling some basic facts about Hall polynomials (see Macdonald~\cite[Sec.~II.1]{MR1354144} for the details). 
Let $R$ be a discrete valuation ring with maximal ideal $\mathfrak{m}$ such that the residue field $R/\mathfrak{m}$ is finite of order $q$. Let $M$ be a finite $R$-module. 
Suppose the maximal ideal $\mathfrak{m}$ of $R$ is generated by an element $\pi$ (a uniformizing element or uniformizer). The structure theorem for finitely generated modules over a principal ideal domain implies that each finite $R$-module $M$ is isomorphic to a direct sum
\begin{align*}
  M\simeq R/(\pi^{\lambda_1})\oplus R/(\pi^{\lambda_2})\oplus \cdots \oplus R/(\pi^{\lambda_\ell}),
\end{align*}
for some integer partition $\lambda=(\lambda_1,\ldots,\lambda_\ell)$, called the type of $M$.

Each submodule of $M$ has type $\mu$ for some partition $\mu\subseteq \lambda$ (defined as $\mu_i\leq \lambda_i$ for each $i\geq 1$). The Hall polynomials $g^\lambda_{\mu\nu}(t)$ are characterized by the property that the number of submodules $N$ of $M$ such that $N$ has type $\mu$ and the quotient $M/N$ has type $\nu$ is given by $g^\lambda_{\mu\nu}(q)$. The polynomial $g^\lambda_{\mu\nu}(t)$ vanishes unless $|\lambda|=|\mu|+|\nu|$ and both $\mu,\nu\subseteq \lambda.$ A product of Hall--Littlewood polynomials can be expressed in the Hall--Littlewood basis as in Butler~\cite[p.~2]{MR1223236}:
\begin{equation}\label{eq:hall-littlewood-product}
P_\mu(x;t) P_\nu(x;t)=\sum_\lambda  g^\lambda_{\mu \nu}(t^{-1})t^{n(\lambda)-n(\mu)-n(\nu)}P_\lambda(x;t),
\end{equation}
where $n(\lambda)=\sum_{i\geq 1} (i-1)\lambda_i$ for each partition $\lambda$.

\begin{example}
  Suppose $\LL=\{(x,\lambda)\}$, corresponding to a nilpotent operator on $\Fq^n$ with Jordan form partition $\lambda$. If $\MM=\{(x,\mu)\}$ for some partition $\mu\subseteq \lambda$ with $|\mu|=k$, then
  \begin{align*}
    \resbrack{\LL}{(n-k),\MM}=\sum_{\nu }g^\lambda_{\mu\nu}(q),
  \end{align*}
  where the sum is taken over all partitions $\nu$. The sum above was studied independently by Delsarte~\cite{MR25463}, Djubjuk~\cite{djubjuk1948number} and Yeh~\cite{MR24428}. It admits a compact product formula: 
  \begin{align*}
    \sum_{\nu }g^\lambda_{\mu\nu}(q)=\prod_{i\geq 1}q^{\mu'_{i+1}(\lambda'_i-\mu'_i)}{\lambda'_i-\mu'_{i+1}\brack \mu'_i-\mu'_{i+1}}_q,
  \end{align*}
where $\lambda'$ denotes the partition conjugate to $\lambda$.
\end{example}
\begin{example}
  Suppose $\LL=\{(f,\lambda)\}$ is primary with $\deg f=d$. This corresponds to an operator $T$ on $\Fq^n$ where the associated $\Fq[x]$-module is isomorphic to a direct sum $\bigoplus_{i\geq 1}    \Fq[x]/(f^{\lambda_{i}})$. If $\MM=\{(f,\mu)\}$  where $\mu\subseteq \lambda$ and $\lVert \MM \rVert=k$, then
  \begin{align*}
    \resbrack{\LL}{(n-k),\MM}=\sum_{\nu }g^\lambda_{\mu\nu}(q^d),
  \end{align*}
since $\Fq[x]/(f^r)$ is isomorphic to $\F_{q^d}[x]/(x^r)$.
\end{example}
Given a pair $\MM = \{ (f_1, \mu^1), \ldots, (f_m, \mu^m) \}$ and $\LL = \{ (f_1, \lambda^1), \ldots, (f_m, \lambda^m) \}$, define 
\( \MM \leq \LL \) if \( \mu^i \subseteq \lambda^i \) for \( 1 \leq i \leq m \). Since the operator part of a restriction is a $T$-invariant submodule of the original operator, its elementary divisors must be contained in those of $\LL$; hence
$$
 \resbrack{\LL}{\mu,\MM} = 0 \quad \text{unless} \quad \MM \leq \LL. 
$$

 By considering all possible restrictions of an operator on $\Fq^n$ with similarity invariants $\LL$ to subspaces of dimension $k$, we obtain the following generalization of the identity in Example \ref{eg:identity}:
\begin{align*}
  \sum_{\substack{(\mu, \MM)\\ \mu_1=n-k\\ |\mu|+\lVert\MM\rVert=n} }\resbrack{\LL}{\mu,\MM} ={n \brack k}_q.
\end{align*}

Let $\widetilde P_\lambda(x;t) = t^{-n(\lambda)}P_\lambda(x;t^{-1})$. Then the duality of $P_\lambda$ and $H_\lambda$ implies
\[
\langle \widetilde H_\mu(x;t), \widetilde P_\nu(x;t) \rangle = \delta_{\mu\nu}.
\]

\begin{definition}\label{def:skew-modified-hl}
The skew modified Hall--Littlewood function $\widetilde H_{\lambda/\mu}(x;t)$ is defined by
\[
\langle \widetilde H_{\lambda/\mu}, \widetilde P_\nu \rangle
:= \langle \widetilde H_\lambda, \widetilde P_\mu\widetilde P_\nu \rangle.\]
\end{definition}
By Equation~\eqref{eq:hall-littlewood-product} we have $\widetilde P_\mu(x;t)\widetilde P_\nu(x;t)= \sum_\lambda g^\lambda_{\mu\nu}(t)\widetilde P_\lambda(x;t).$ It follows that 
\[
 \langle \widetilde H_{\lambda/\mu}(x;t), \widetilde P_\nu(x;t) \rangle= g^\lambda_{\mu\nu}(t).
\]
Therefore
\[
 \widetilde H_{\lambda/\mu}(x;t)
= \sum_\nu g^\lambda_{\mu\nu}(t)\widetilde H_\nu(x;t).
\]

Note that $\widetilde{H}_{\lambda/\mu}(x;t)$ vanishes unless $\mu\subseteq \lambda.$ When $\mu$ is the empty partition, we have $\widetilde{H}_{\lambda/\mu}(x;t)=\widetilde{H}_\lambda(x;t)$. On the other hand $\widetilde{H}_{\lambda/\mu}(x;t)=1$ if $\mu=\lambda$.

\begin{conjecture}\label{conj:skewHLschurpositive}
 In the Schur expansion $\widetilde{H}_{\lambda/\mu}(x;t)=\sum_\nu a^\nu_{\lambda\mu}(t)s_\nu$, the coefficients $a^\nu_{\lambda\mu}(t)$ are nonnegative integer polynomials in $t$.
\end{conjecture}
\begin{problem}\label{prob:skewHLschurcoeffs}
Find a combinatorial formula for the coefficients in the Schur expansion of $\widetilde{H}_{\lambda/\mu}(x;t).$
\end{problem}
A solution to Problem \ref{prob:skewHLschurcoeffs} would be significant as the special case $\mu=\emptyset$ would recover the celebrated cocharge formula of Lascoux--Sch{\"u}tzenberger~\cite{LascouxSchutzenberger1978}.

\begin{theorem}\label{thm:coeffcomputation}
  If $\LL=\{(f_1,\lambda^1),\ldots,(f_r,\lambda^r)\}$ and $\MM=\{(f_1,\mu^1),\ldots,(f_r,\mu^r)\}$ where $\deg f_i=d_i$ for $1\leq i\leq r$, then
  \begin{align*}
    \resbrack{\LL}{\mu,\MM}=\langle A_{\LL\MM}(x;q), B_{\mu}(x;q)\rangle, 
  \end{align*}
  where
  \begin{align*}
    A_{\LL\MM}(x;t)&:=\prod_{i=1}^r p_{d_i}\circ \widetilde{H}_{\lambda^i/\mu^i}(x;t), \\
    B_{\mu}(x;t)&:=(-1)^{\sum_{j\geq 2}\mu_j}t^{\sum_{j\geq 2}{\mu_j\choose 2}}\widetilde{W}_{\mu}(x;t).
  \end{align*}
\end{theorem}
\begin{proof}
  Let $T\in \CC(\LL)$ and suppose $\lVert \LL \rVert=n$. By definition, $\resbrack{\LL}{\mu,\MM}$ is equal to the number of subspaces $W$ such that the restriction $T_W:W\to V$ has similarity invariants $(\mu,\MM)$.  
Consider a $T$-invariant subspace $U\subset V$ such that
  \begin{enumerate}
  \item the restriction $T_U$ of $T$ to $U$ has similarity invariants $\MM$, 
  \item the map  $\hat{T}_U:V/U\to V/U$ defined by $\hat{T}_U(\alpha+U)=T\alpha+U$ has similarity invariants $\NN=\{(f_1,\nu^1),\ldots,(f_r,\nu^r)\}$.
  \end{enumerate}
The number of such subspaces $U$ is given by
  \begin{align*}
    \prod_{i=1}^r g^{\lambda^i}_{\mu^i \nu^i}(q^{d_i}).
  \end{align*}
  Once $U$ has been chosen, a subspace $W$ counted by $\resbrack{\LL}{\mu,\MM}$ is characterized by the property that $W\supseteq U$ and the restriction of $\hat{T}_U$ to $W/U$ has similarity invariants $(\mu,\emptyset)$. By Lemma \ref{lem:restrictionformula} the number of such subspaces is given by 
  \begin{align*}
    \resbrack{\NN}{\mu,\emptyset}=\langle F_{\NN}(x;q),B_\mu(x;q) \rangle,
  \end{align*}
  where  $B_\mu(x;t)=(-1)^{\sum_{j\geq 2}\mu_j}t^{\sum_{j\geq 2}{\mu_j\choose 2}}\widetilde{W}_{\mu}(x;t).$ Therefore,
  \begin{align*}
    \resbrack{\LL}{\mu,\MM}&=\sum_{(\nu^1,\ldots,\nu^r)} \langle F_{\NN}(x;q),B_\mu(x;q) \rangle \prod_{i=1}^r g^{\lambda^i}_{\mu^i \nu^i}(q^{d_i}),
  \end{align*}
  where the sum is over all tuples $(\nu^1,\ldots,\nu^r)$ of partitions with $\NN=\{(f_1,\nu^1),\ldots,(f_r,\nu^r)\}$. Since 
$    F_{\NN}(x;t)=\prod_{i=1}^r p_{d_i}[\widetilde{H}_{\nu^i}(x;t)]$, it follows that 
  \begin{align*}
    \resbrack{\LL}{\mu,\MM}&=\langle A_{\LL\MM}(x;q), B_\mu(x;q) \rangle,
  \end{align*}
  where
  \begin{align*}
    A_{\LL\MM}(x;t)&= \sum_{(\nu^1,\ldots,\nu^r)} \prod_{i=1}^r g^{\lambda^i}_{\mu^i \nu^i}(t^{d_i})\; p_{d_i}[\widetilde{H}_{\nu^i}(x;t)]\\
    &=\prod_{i=1}^r \sum_{\nu^i} g^{\lambda^i}_{\mu^i \nu^i}(t^{d_i})\; p_{d_i}[\widetilde{H}_{\nu^i}(x;t)]\\
    &=\prod_{i=1}^r p_{d_i}\big[\sum_{\nu^i} g^{\lambda^i}_{\mu^i \nu^i}(t) \widetilde{H}_{\nu^i}(x;t)\big]\\
    &=\prod_{i=1}^r p_{d_i}\circ\widetilde{H}_{\lambda^i/ \mu^i}(x;t), 
  \end{align*}
  proving the result.
\end{proof}
\begin{remark}
The symmetric function $A_{\LL\MM}(x;t)$ is homogeneous of degree $\lVert \LL \rVert-\lVert \MM \rVert$ and vanishes unless $\MM\leq \LL$. We have $A_{\LL\emptyset}(x;t)=F_\LL(x;t)$.
\end{remark}
Theorem \ref{thm:coeffcomputation} together with Proposition \ref{prop:eta} yields the following result.
\begin{theorem}\label{thm:eta}
Let $\LL$ index a conjugacy class in $\MM_n(\Fq)$, and let $(\mu,\MM)$ be the similarity invariants of a linear map defined on a subspace of $\Fq^n$. Then
  \begin{align*}
    \eta((\mu,\MM),\LL)=\frac{1}{{n \brack \mu_1}_q}\frac{|\CC(\LL)|}{\pi(\mu,\MM)} \langle A_{\LL\MM}(x;q), B_{\mu}(x;q)\rangle.
  \end{align*}
\end{theorem}

\begin{corollary}\label{cor:etasimple}
If $\LL$ indexes a conjugacy class in $\MM_n(\Fq)$ and $\mu$ is a partition of $n$, then
  \begin{align*}
    \eta((\mu,\emptyset),\LL)=\frac{1}{{n \brack \mu_1}_q}\frac{|\CC(\LL)|}{\pi(\mu,\emptyset)} \langle F_{\LL}(x;q), B_{\mu}(x;q)\rangle.
  \end{align*}
\end{corollary}

Let $\MM_{n,k}(\Fq)$ denote the space of $n\times k$ matrices over $\Fq$. The similarity invariants of a matrix $A\in \MM_{n,k}(\Fq)$ are defined to be the similarity invariants of the corresponding linear transformation from $\Fq^k$ to $\Fq^n$. Theorem \ref{thm:eta} can be recast in the setting of matrices as follows.
\begin{theorem}\label{thm:eta-matrix}
 Suppose $k\leq n$ and let $A\in \MM_{n,k}(\Fq)$ have similarity invariants $(\mu,\MM)$. Then the number of matrices in $\MM_n(\Fq)$ which have similarity invariants $\LL$ and whose first $k$ columns are given by $A$ equals   $\eta((\mu,\MM),\LL)$.
\end{theorem}
We consider applications of this theorem in the next sections.

\section{Induced representations and irreducible unipotent characters}\label{sec:representation-theory}

The extension counts above also have a natural representation-theoretic interpretation. In this section we explain how orbit intersections determine induced character values, and how Schur expansions connect the invariant flag generating function with irreducible unipotent characters.

  \subsection{Orbit intersections and induced characters}\label{subsec:orbit-intersections-induced-characters}
Let $m<n$ be a positive integer and suppose $H$ denotes the subgroup of $G={\rm GL}_n(\Fq)$ consisting of matrices of the form
  \begin{align*}
    \begin{pmatrix}
      I_m & A_1\\
      {\bf 0} & A_2
    \end{pmatrix},
  \end{align*}
	  where $I_m$ denotes the $m\times m$ identity matrix, $A_1$ is $m\times (n-m)$ and $A_2\in {\rm GL}_{n-m}(\Fq)$. Then $H$ is isomorphic to the semidirect product $\MM_{m,n-m}(\Fq)\rtimes {\rm GL}_{n-m}(\Fq)$, where ${\rm GL}_{n-m}(\Fq)$ acts on the additive group $\MM_{m,n-m}(\Fq)$ by right multiplication. Let $\ch(\Ind^G_H 1)$ denote the character of the induced representation of the trivial representation from $H$ to $G$. For $g\in G$ let $Z_G(g)$ denote the centralizer of $g$ in $G$ and write ${\rm Cl}_G(g)$ for the conjugacy class of $g$ in $G$. Then
	  \begin{align*}
	    \ch(\Ind^G_H 1)(g)
	    &=\#\{xH\in G/H:g xH=xH\} \\
	    &=\frac{\#\{x\in G:x^{-1}gx\in H\}}{|H|} \\
	    &=\frac{|Z_G(g)|}{|H|}|{\rm Cl}_G(g)\cap H|.
	  \end{align*}
	  Thus the intersection size $|{\rm Cl}_G(g)\cap H|$, which can be computed by Theorem~\ref{thm:eta}, determines the induced character value up to an explicit factor.

\subsection{Schur expansion and irreducible unipotent characters}\label{subsec:schur-unipotent}
Let $B$ denote the Borel subgroup of $G$ consisting of upper triangular matrices.

\begin{definition}\label{def:unipotent-characters}
The irreducible unipotent characters of $G$ are the irreducible constituents of the character of the permutation representation of $G$ acting on $G/B$. Let $\chi^\lambda_G$ denote the irreducible unipotent character indexed by $\lambda$.
\end{definition}
For instance, $\chi_G^{(n)}$ denotes the trivial character while $\chi_G^{(1^n)}$ denotes the Steinberg character.

\begin{definition}\label{def:permutation-characters}
Let $\alpha=(\alpha_1,\ldots,\alpha_r)$ be a composition of $n$ and let $P_\alpha\leq G$ be the standard parabolic subgroup stabilizing a partial flag of type $\alpha$:
\begin{align*}
0=V_0\subset V_1\subset \cdots \subset V_r=\Fq^n,\quad \dim(V_i/V_{i-1})=\alpha_i.
\end{align*}
Let
\begin{align*}
\xi_\alpha:=\Ind^G_{P_\alpha}1
\end{align*}
be the permutation character of $G={\rm GL}_n(\Fq)$ acting on the coset space $G/P_\alpha$ (equivalently, on partial flags of type $\alpha$). Note that for $g\in G$, the value $\xi_\alpha(g)$ equals the number of $g$-stable flags of type $\alpha$ which only depends on the partition obtained by sorting the parts of $\alpha$ in weakly decreasing order.
\end{definition}

The Kostka numbers $K_{\lambda\mu}$ appear as transition coefficients in the Schur to monomial and complete homogeneous to Schur expansions:
\begin{align*}
s_\lambda=\sum_\mu K_{\lambda\mu}m_\mu \quad \text{and} \quad h_\mu=\sum_\lambda K_{\lambda\mu}s_\lambda.
\end{align*}
We require the following decomposition of permutation characters into irreducible unipotent characters; see Ernst and Schmidt~\cite[Lemma~2.3]{ErnstSchmidt2024}.
\begin{lemma}\label{lem:xi-kostka}
For each partition $\mu$ of $n$, we have
\begin{align*}
\xi_\mu=\sum_{\lambda\vdash n}K_{\lambda\mu}\chi^\lambda_G,
\end{align*}
where $K_{\lambda\mu}$ denotes a Kostka number and $\chi^\lambda_G$ denotes the irreducible unipotent character of $G$ indexed by $\lambda$.
\end{lemma}

\begin{remark}
Ernst and Schmidt use the convention in James~\cite{James1986} for irreducible unipotent characters in which the trivial character is indexed by $(n)$; this agrees with the convention used here. In Macdonald~\cite{MR1354144}, the partition indexing the character is conjugated.
\end{remark}

\begin{proposition}\label{prop:schur-invflag}
For each $g\in G$, we have the following Schur expansion of the invariant flag generating function:
\begin{align*}
F_g(x)=\sum_\lambda \chi^\lambda_G(g)s_\lambda.
\end{align*}
\end{proposition}
\begin{proof}
By Definition~\ref{def:invflag},
\begin{align*}
F_g(x)=\sum_\mu \xi_\mu(g)m_\mu.
\end{align*}
By Lemma~\ref{lem:xi-kostka}, we have $\xi_\mu(g)=\sum_\lambda K_{\lambda\mu}\chi^\lambda_G(g)=\sum_\lambda\langle h_\mu,s_\lambda\rangle\chi^\lambda_G(g)$. Therefore
\begin{align*}
F_g(x)&=\sum_\mu\sum_\lambda \langle h_\mu,s_\lambda\rangle\chi^\lambda_G(g)m_\mu\\
&=\sum_\mu\left\langle h_\mu,\sum_\lambda\chi^\lambda_G(g)s_\lambda\right\rangle m_\mu\\
&=\sum_\lambda\chi^\lambda_G(g)s_\lambda,
\end{align*}
where the last equality follows from the fact that the monomial and complete homogeneous bases are dual with respect to the Hall scalar product.
\end{proof}
When $g$ is unipotent with Jordan form partition $\mu$ we have $F_g(x)=\tilde{H}_\lambda(x;q)$ and in this case $\chi^\lambda_G(g)$ coincides with a modified Kostka--Foulkes polynomial $\tilde{K}_{\lambda\mu}(q)$. This was already observed by Lusztig \cite[Eq.\ 2.2]{Lusztig1981}.
\section{Applications to matrix enumeration}\label{sec:matrixcounting}

\subsection{Simple linear transformations}\label{subsec:simple-linear-transformations}
Recall that a linear map $T$ defined on a subspace of $V$ is simple if each $T$-invariant subspace is either the zero subspace or all of $V$.

\begin{example}
  The linear map $T$ defined on the subspace $W={\rm span}\{e_i\}_{1\leq i\leq 8}$ of $\Fq^n(n\geq 11)$  by
  \begin{align*}
    e_1\xrightarrow{T} e_2 \xrightarrow{T}e_3 \xrightarrow{T}e_4&\xrightarrow{T}e_9  \\
    e_5\xrightarrow{T} e_6  \xrightarrow{T}e_7 &\xrightarrow{T}e_{10} \\
e_8 &\xrightarrow{T}e_{11}
  \end{align*}
is simple with defect dimensions $\mu=(n-8,3,2,2,1)$.
\end{example}
The above example can be generalized to obtain the following result.
\begin{proposition}\label{prop:defchar}
  Given a partition $\lambda=(\lambda_1,\ldots,\lambda_\ell)$, let $\{e_{ij}\}(1\leq i\leq \ell,1\leq j\leq \lambda_i+1)$ be a linearly independent subset of vectors in $V=\Fq^n$. Let $W$ denote the span of the vectors $\{e_{ij}\}$ for $1\leq i\leq \ell$ and $1\leq j\leq \lambda_i$. Then the linear map $T\in L(W,V)$ whose action on the basis of $W$ is specified by $Te_{ij}=e_{i,j+1}$ is simple with defect dimensions $(n-|\lambda|,\lambda'_1,\lambda'_2,\ldots)$.
\end{proposition}
\begin{proof}
Let $W_j=W\cap T^{-1}W\cap\cdots\cap T^{-j}W$. A basis vector $e_{ab}\in W$ lies in $W_j$ precisely when its first $j$ iterates under $T$ remain in $W$, equivalently when $b+j\leq \lambda_a$. Thus $W_j$ is spanned by the vectors $e_{ab}$ with $1\leq b\leq \lambda_a-j$, and hence
\[
\dim W_j=\sum_{a=1}^\ell \max(\lambda_a-j,0)=|\lambda|-(\lambda'_1+\cdots+\lambda'_j).
\]
It follows that the successive defect dimensions are $n-|\lambda|,\lambda'_1,\lambda'_2,\ldots$. Since these intersections eventually become zero, the map $T$ has no operator part and is simple.
\end{proof}

\subsection{Matrix polynomials with prescribed Smith form}\label{subsec:matrix-polynomials}
We give an application of our results to matrix polynomials over a finite field. Given positive integers $m$ and $d$, define
\begin{align*}
  \MM_m(d;q):=\{x^dI+x^{d-1}C_{d-1}+\cdots +C_0:C_i \in \MM_m(\Fq) \mbox{ for }0\leq i\leq d-1\}.
\end{align*}
Here $I$ denotes the identity matrix. Elements of $\MM_m(d;q)$ are polynomials whose coefficients are matrices, often referred to as \emph{matrix polynomials}. Any matrix polynomial in $\MM_m(d;q)$ can also be viewed as an element of the algebra $\MM_m(\Fq[x])$. Let $\GL_m(\Fq[x])$ denote the subset of all elements of $\MM_m(\Fq[x])$ whose determinant is a nonzero scalar in $\Fq$. The elements of $\GL_m(\Fq[x])$ are the \emph{unimodular} polynomial matrices, and they form a group under matrix multiplication. Two matrix polynomials $P,Q\in \MM_m(d;q)$ are said to be equivalent if there exist elements $g_1,g_2\in \GL_m(\Fq[x])$ such that $g_1Pg_2=Q$. Each element $P\in \MM_m(d;q)$ is equivalent to a diagonal matrix ${\rm diag}(p_m,\ldots,p_1)$ where $p_j(1\leq j\leq m)$ are monic polynomials satisfying $p_{j+1}\mid p_{j}$ for $1\leq j<m$. This diagonal form is called the Smith normal form of $P$ while the polynomials $p_1,\ldots,p_m$ are called the invariant factors or Smith invariants of $P.$ Note that for any matrix polynomial in $\MM_m(d;q),$ the product of diagonal entries in its Smith form has degree $md$.

\begin{remark}
If we write $G=\GL_m(\Fq[x])$, then we have an action of the direct product  $G\times G$ on the set $\{g_1Pg_2: P\in \MM_m(d;q)\mbox{ and }g_1,g_2\in G\}$ defined by $(g_1,g_2)*P=g_1Pg_2^{-1}$. The orbit of an element $P$ under this action consists of all polynomial matrices equivalent to $P$. Note that $G$ has infinite order for $m>1$ and, consequently, the orbits are infinite in general.
\end{remark}
We are interested in the following problem.
\begin{problem}\label{prob:matrix-polynomial-smith}
Determine the number of polynomial matrices in $\MM_m(d;q)$ with a given Smith form.  
\end{problem}
For $d=1$, this problem reduces to determining the number of square matrices in a conjugacy class, since matrices $A,B\in \MM_m(\Fq)$ are similar if and only if $xI-A$ is equivalent to $xI-B$. 
Given a matrix $A\in \MM_m(\Fq)$ with similarity invariants $\LL=\{(f_1,\lambda^1),\ldots,(f_r,\lambda^r)\}$, the Smith form of $xI-A$ is the polynomial matrix ${\rm diag}(p_m,\ldots,p_1)$, where $p_j=\prod_{i=1}^r f_i^{\lambda^i_j}$ for $1\leq j\leq m$. This correspondence defines a bijection $\Theta_m$ from the functions $\LL$ which index conjugacy classes of $m\times m$ matrices to the $m$-tuples $(p_m\mid p_{m-1}\mid \cdots \mid p_1)$ of monic polynomials over $\Fq$ such that the product $p_1\cdots p_m$ has degree $m.$

\begin{theorem}\label{thm:matrix-polynomial-smith}
  Let $(p_m\mid\cdots\mid p_1)$ be an $m$-tuple of monic polynomials in $\Fq[x]$, with product of degree $n=md$. The number of elements in $\MM_m(d;q)$ whose Smith form has diagonal entries $p_j$ for $1\leq j\leq m$ is
  \begin{align*}
    \eta((\mu,\emptyset),\LL),
  \end{align*}
  where $\LL$ is the unique function for which $\Theta_n(\LL)=(\underbrace{1,\ldots,1}_{n-m},p_m,\ldots,p_1)$, and $\mu=(m^d)$ denotes the rectangular partition with $d$ parts equal to $m$.
\end{theorem}
\begin{proof}
  Let $N(p_1,\ldots,p_m)$ denote the number of elements in $\MM_m(d;q)$ whose Smith form has diagonal entries $p_j(1\leq j\leq m)$. It is easy to verify (see Gohberg, Lancaster and Rodman~\cite[Thm.~1.1]{MR3396732}) that the matrix polynomial $P=x^dI+x^{d-1}C_{d-1}+\cdots+C_0$ has the same nonunit invariant factors as $xI-A$ where $A$ is the block companion matrix
  \begin{equation}\label{eq:bcm}
    \begin{bmatrix}
0 & 0 & 0 & \cdots & 0 & -C_0 \\
I & 0 & 0 & \cdots & 0 & -C_1 \\
0 & I & 0 & \cdots & 0 & -C_2 \\
\vdots & \vdots & \vdots & \ddots & \vdots & \vdots \\
0 & 0 & 0 & \cdots & I & -C_{d-1}
\end{bmatrix}.
  \end{equation}
Here $0$ and $I$ denote the zero and identity matrices of size $m$. Therefore $N(p_1,\ldots,p_m)$ is equal to the number of $n\times n$ matrices of the above form that lie in the conjugacy class indexed by
\[
\LL=\Theta_n^{-1}((\underbrace{1,\ldots,1}_{n-m},p_m,\ldots,p_1)).
\]
This number is $\eta((\mu,\emptyset),\LL)$, where $(\mu,\emptyset)$ are the similarity invariants of the map $T:\Fq^{m(d-1)}\to \Fq^{md}$ whose matrix, with respect to the standard bases, consists of the first $m(d-1)$ columns of Equation~\eqref{eq:bcm}. Since $Te_i=e_{i+m}$ for each $1\leq i\leq m(d-1)$, successive applications of $T$ to $e_i$, for $1\leq i\leq m$, give the chains
$$
e_i\to e_{i+m}\to e_{i+2m}\to\cdots \to e_{i+(d-1)m}.
$$
By Proposition~\ref{prop:defchar}, the map $T$ is simple with defect dimensions $\mu=(m^d)$, and the result follows. 
\end{proof}

\subsection{Matrices with prescribed columns and characteristic polynomial}\label{subsec:prescribed-columns-characteristic}
The counting formula of Theorem \ref{thm:eta} can be used to enumerate various classes of matrices with some specified columns (or rows); for instance, matrices with a given rank, or minimal or characteristic polynomial, or other conjugacy class invariants. We illustrate this by computing the number of matrices with specified columns and characteristic polynomial.

\begin{definition}\label{def:characteristic-extension-count}
Let $T$ be a linear map defined on a subspace $W$ of $V$ with similarity invariants $(\mu,\MM)$. Given a monic polynomial $f\in \Fq[x]$, let $\Phi((\mu,\MM),f)$ denote the number of extensions of $T$ to all of  $V$ which have characteristic polynomial $f$.
\end{definition}
The characteristic polynomial of an operator with similarity invariants $\LL=\{(f_1,\lambda^1),\ldots,(f_r,\lambda^r)\}$ is $\prod_{i=1}^r f_i^{|\lambda^i|}$. Therefore, if $f=\prod_{i=1}^r f_i^{n_i}$ denotes the canonical factorization of $f$ into distinct irreducible polynomials $f_i$, then
\begin{equation}\label{eq:phiviaeta}
  \Phi((\mu,\MM),f)=\sum_{\LL}\eta((\mu,\MM),\LL),
\end{equation}
where the sum is taken over all $\LL=\{(f_1,\lambda^1),\ldots,(f_r,\lambda^r)\}$ satisfying $|\lambda^i|=n_i$ for $1\leq i\leq r.$ We will see that the sum above admits a nice closed form in the case $\MM=\emptyset$. Recall the definition of $c_q(d,\lambda)$ from Equation \eqref{eq:centralizer}. We require the following lemma.

\begin{lemma}\label{lem:censum}
  We have
  \begin{align*}
    \sum_{\lambda \vdash n}\frac{\widetilde{H}_\lambda(x;t)}{c_t(1,\lambda)}=h_n\left[\frac{X}{t-1}\right],
  \end{align*}
    where $\widetilde{H}_\lambda(x;t)$ denotes a modified Hall--Littlewood polynomial.
  \end{lemma}
    \begin{proof}
    We begin with the identity (Macdonald~\cite[Example~III.6.1]{MR1354144})
    \begin{align*}
      h_n=\sum_{\lambda \vdash n}t^{n(\lambda)}P_\lambda(x;t),
    \end{align*}
    where $n(\lambda)=\sum_{i\geq 1}(i-1)\lambda_i$ and $P_\lambda$ denotes a Hall--Littlewood polynomial. From the plethystic relation
    \begin{align*}
      H_\lambda(x;t)=\left(\prod_{i\geq 1}(t;t)_{m_i(\lambda)}\right) P_\lambda\left[\frac{X}{1-t}\right],
    \end{align*}
    it follows that 
    \begin{align*}
      h_n\left[\frac{X}{1-t}\right]&=\sum_{\lambda \vdash n}t^{n(\lambda)}P_\lambda\left[\frac{X}{1-t}\right]\\
      &=\sum_{\lambda \vdash n}t^{n(\lambda)}\frac{H_\lambda(x;t)}{\prod_{i\geq 1}(t;t)_{m_i(\lambda)}}.
    \end{align*}
    Replacing $t$ by $1/t$, we obtain
    \begin{align*}
      h_n\left[\frac{tX}{t-1}\right]&=\sum_{\lambda \vdash n}t^{-n(\lambda)}\frac{H_\lambda(x;t^{-1})}{\prod_{i\geq 1}(t^{-1};t^{-1})_{m_i(\lambda)}}\\
      &=\sum_{\lambda \vdash n}t^{-2n(\lambda)}\frac{\widetilde{H}_\lambda(x;t)}{\prod_{i\geq 1}(t^{-1};t^{-1})_{m_i(\lambda)}}.
    \end{align*}
    Therefore,
    \begin{align*}
            h_n\left[\frac{X}{t-1}\right]&= \sum_{\lambda \vdash n}t^{-n-2n(\lambda)}\frac{\widetilde{H}_\lambda(x;t)}{\prod_{i\geq 1}(t^{-1};t^{-1})_{m_i(\lambda)}}.
    \end{align*}
    Since $n+2n(\lambda)=\sum_{i\geq 1} \lambda'_i+2{\lambda'_i \choose 2}=\sum_{i\geq 1}(\lambda'_i)^2$, it follows that
        \begin{align*}
          h_n\left[\frac{X}{t-1}\right]&= \sum_{\lambda \vdash n}\frac{\widetilde{H}_\lambda(x;t)}{t^{\sum_{i\geq 1}(\lambda'_i)^2}\prod_{i\geq 1}(t^{-1};t^{-1})_{m_i(\lambda)}}\\
          &=\sum_{\lambda \vdash n}\frac{\widetilde{H}_\lambda(x;t)}{c_t(1,\lambda)},
    \end{align*}
completing the proof.
  \end{proof}
  \begin{theorem}\label{thm:characteristic-extension-count}
If $f=\prod_{i=1}^r f_i^{n_i}$ where the $f_i$ are distinct irreducible polynomials over $\Fq$ with $\deg f_i=d_i$ and $\mu$ is a partition of $n$, then
    \begin{align*}
      \Phi((\mu,\emptyset),f)=\frac{\gamma_q(n)}{{n \brack \mu_1}_q\pi(\mu,\emptyset)}\langle C_{f}(x;q),B_\mu(x;q)   \rangle, 
    \end{align*}
    where
    \begin{align*}
    C_f(x;t):=\prod_{i=1}^r p_{d_i}\circ h_{n_i}\left[ \frac{X}{t-1} \right].      
    \end{align*}
  \end{theorem}
  \begin{proof}
From Equation \eqref{eq:phiviaeta} and Corollary \ref{cor:etasimple}, 
    \begin{align*}
        \Phi((\mu,\emptyset),f)=\frac{1}{{n \brack \mu_1}_q \pi(\mu,\emptyset)} \sum_{\LL}|\CC(\LL)| \langle F_{\LL}(x;q), B_{\mu}(x;q)\rangle,
    \end{align*}
    where the sum is taken over all functions $\LL=\{(f_1,\lambda^1),\ldots,(f_r,\lambda^r)\}$ satisfying $|\lambda^i|=n_i$ for $1\leq i\leq r.$ If we write $c(\LL)=\prod_{i=1}^r c_q(d_i,\lambda^i)$, then
    \begin{align*}
      \Phi((\mu,\emptyset),f)&=\frac{\gamma_q(n)}{{n \brack \mu_1}_q \pi(\mu,\emptyset)} \sum_{\LL} \langle \frac{F_{\LL}(x;q)}{c(\LL)}, B_{\mu}(x;q)\rangle\\
      &=\frac{\gamma_q(n)}{{n \brack \mu_1}_q \pi(\mu,\emptyset)} \langle C_{f}(x;q), B_{\mu}(x;q)\rangle,
    \end{align*}
    where, with $\LL$ varying as before,
    \begin{align}
      C_f(x;t)&=\sum_{\LL}\frac{F_{\LL}(x;t)}{\prod_{i=1}^r c_t(d_i,\lambda^i) }\label{eq:cf}\\ 
                    &=\sum_{\substack{\lambda^i \vdash n_i\\ 1\leq i\leq r}}\prod_{i=1}^r p_{d_i}\bigg[ \frac{\widetilde{H}_{\lambda^i}(x;t)}{c_t(1,\lambda^i)} \bigg] \notag\\ 
              &=\prod_{i=1}^r p_{d_i}\bigg[\sum_{\lambda^i\vdash n_i} \frac{\widetilde{H}_{\lambda^i}(x;t)}{c_t(1,\lambda^i)} \bigg], \notag\\
      &=\prod_{i=1}^r p_{d_i}\circ h_{n_i}\left[ \frac{X}{t-1} \right].\notag
    \end{align}
Here the last equality follows from Lemma \ref{lem:censum}.
  \end{proof}
  \begin{corollary}\label{cor:matrix-polynomial-determinant}
Let $m$ and $d$ be positive integers, and let $f=\prod_{i=1}^r f_i^{n_i}$ be a monic polynomial of degree $md$, where the $f_i$ are distinct irreducible polynomials over $\Fq$ and $d_i=\deg f_i$. The number of matrix polynomials $P\in \MM_m(d;q)$ with determinant $f$ is
    \begin{align*}
      \gamma_q(m)\langle C_f(x;q),B_{(m^d)}(x;q)\rangle,
    \end{align*}
    where
    \begin{align*}
      C_f(x;t):=\prod_{i=1}^r p_{d_i}\circ h_{n_i}\left[\frac{X}{t-1}\right].
    \end{align*}
  \end{corollary}
  \begin{proof}
Let $n=md$ and $\mu=(m^d)$. As in the proof of Theorem~\ref{thm:matrix-polynomial-smith}, each polynomial $P=x^dI+x^{d-1}C_{d-1}+\cdots+C_0\in \MM_m(d;q)$ determines the block companion matrix in Equation~\eqref{eq:bcm}. This gives a bijection from $\MM_m(d;q)$ to the set of extensions of the partial map $T$ defined by the first $m(d-1)$ columns of Equation~\eqref{eq:bcm}. By Proposition~\ref{prop:defchar}, the map $T$ is simple with defect dimensions $\mu=(m^d)$, so its similarity invariants are $(\mu,\emptyset)$. Moreover, the characteristic polynomial of the block companion matrix is $\det P$. Therefore the number of elements of $\MM_m(d;q)$ with determinant $f$ is exactly $\Phi((\mu,\emptyset),f)$.

Applying Theorem~\ref{thm:characteristic-extension-count} gives
    \begin{align*}
      \Phi((\mu,\emptyset),f)=\frac{\gamma_q(md)}{{md \brack m}_q\pi((m^d),\emptyset)}\langle C_f(x;q),B_{(m^d)}(x;q)\rangle.
    \end{align*}
By Proposition~\ref{prop:partialsize},      $\pi((m^d),\emptyset)=\gamma_q(m(d-1))q^{m^2(d-1)}.$ Using the identity
    \begin{align*}
      {md \brack m}_q=\frac{\gamma_q(md)}{\gamma_q(m)\gamma_q(m(d-1))q^{m^2(d-1)}},
    \end{align*}
the prefactor simplifies to $\gamma_q(m)$, as claimed. 
  \end{proof}

\begin{remark}
The case of Corollary~\ref{cor:matrix-polynomial-determinant} in which $f$ is irreducible is connected to a problem of Niederreiter arising from pseudorandom number generation. Through the block companion construction in the proof, counting matrix polynomials in $\MM_m(d;q)$ with a fixed irreducible determinant is equivalent to counting block companion matrices with a fixed irreducible characteristic polynomial. For the original motivation and the solution, see Niederreiter~\cite{MR1334623} and Chen--Tseng~\cite{MR3093853}, respectively.
\end{remark}

  The following result was originally proved independently by Gerstenhaber~\cite{Mg1961} and Reiner~\cite{Ir1961}. 
  \begin{corollary}
    For each monic polynomial $f=\prod_{i=1}^r f_i^{n_i}$ of degree $n$, the number of matrices in $\MM_n(\Fq)$ with characteristic polynomial $f$ is given by
    \begin{equation*}
      \frac{\gamma_q(n)q^{-n}}{\prod_{i=1}^r (q^{-d_i};q^{-d_i})_{n_i}},
    \end{equation*}
    where $d_i=\deg f_i$ for $1\leq i\leq r$.
  \end{corollary}
  \begin{proof}
    Taking $d=1$ in Corollary~\ref{cor:matrix-polynomial-determinant}, the number of $n\times n$ matrices with characteristic polynomial $f$ equals
    \begin{align*}
      \gamma_q(n)\langle C_f(x;q),B_{(n)}(x;q)\rangle.
    \end{align*}
Since $B_{(n)}(x;t)=\widetilde{W}_n(x;t)=h_n$, it remains to compute $\langle C_f(x;t),h_n \rangle$, which is equal to the coefficient of $m_n(x)$ in the monomial expansion of $C_f(x;t)$ (since $m_\lambda$ and $h_\lambda$ are dual bases). Since $C_f(x;t)=\prod_{i=1}^rC_{f_i^{n_i}}(x;t)$, it can be seen, by comparing coefficients, that
    \begin{align*}
      \langle C_f(x;t),h_n \rangle&=\prod_{i=1}^r   \langle C_{f_i^{n_i}}(x;t),h_{d_i n_i} \rangle.
    \end{align*}
    We claim that
    \begin{equation}\label{eq:principal-specialization-plethysm}
     \langle    p_{d}\circ h_{k}\left[ \frac{X}{t-1}  \right],h_{kd} \rangle=\frac{t^{-kd}}{(t^{-d};t^{-d})_k}. 
    \end{equation}
    To see this, begin with the identity $h_k=\sum_{\lambda \vdash k}p_\lambda/z_\lambda$, where $z_\lambda=\langle p_\lambda,p_\lambda \rangle=\prod_{i\geq 1}i^{m_i(\lambda)}m_i(\lambda)!$. We have
    \begin{align*}
      h_k\left[ \frac{X}{t-1} \right]=\sum_{\lambda \vdash k}\frac{1}{z_\lambda}\frac{1}{\prod_{j\geq 1}(t^{\lambda_j}-1)}p_\lambda.
    \end{align*}
    Therefore
    \begin{align*}
     p_d\circ h_k\left[ \frac{X}{t-1} \right]=\sum_{\lambda \vdash k}\frac{1}{z_\lambda}\frac{1}{\prod_{j\geq 1}(t^{d\lambda_j}-1)}p_{d\lambda},
    \end{align*}
    where $d\lambda$ denotes the partition $(d\lambda_1,d\lambda_2,\ldots)$. The coefficient of $m_{kd}$ in the monomial expansion of the sum above is given by
    \begin{align*}
      \sum_{\lambda \vdash k}\frac{1}{z_\lambda}\frac{1}{\prod_{j\geq 1}(t^{d\lambda_j}-1)}&=\sum_{\lambda \vdash k}\frac{1}{z_\lambda}\frac{t^{-kd}}{\prod_{j\geq 1}(1-t^{-d\lambda_j})}\\
                                                                                           &=t^{-kd}\sum_{\lambda \vdash k}\frac{1}{z_\lambda}p_\lambda(1,t^{-d},t^{-2d},\ldots)\\
                                                                                           &=t^{-kd}h_k(1,t^{-d},t^{-2d},\ldots)\\
      &=\frac{t^{-kd}}{(t^{-d};t^{-d})_k},
    \end{align*}
    where the last equality follows from the stable principal specialization identity for $h_k$ (see Stanley~\cite[Prop.~7.8.3]{MR1676282}). The result now follows from the claim and the fact that $\sum n_id_i=n.$
  \end{proof}
\section{Matrices with given support in a conjugacy class}\label{sec:supports-conjugacy}
In this section we consider the problem of counting matrices in a conjugacy class with some entries specified to be zero. More precisely, given a subset $S\subseteq [n]\times [n]$ and a conjugacy class $\CC$ of $\MM_n(\Fq)$, we would like to enumerate the number of matrices $A=(a_{ij})\in \CC$ such that $a_{ij}=0$ whenever $(i,j)\notin S$. Let $N(\CC;S)$ denote the number of such matrices; this is $N(\CC;S;0)$ in the notation of Definition~\ref{def:prescribed-count}. To each such support $S$, we associate a directed graph $D=D_S$ with vertex set $[n]$ and edge set $S$. Let $A_S$ denote the adjacency matrix of $D_S$.
\begin{example}
  Let $n=5$ and consider the support
  \[
S = \{ (1,2),\ (1,5),\ (2,3),\ (3,1),\ (4,4),\ (5,2),\ (5,5) \}.
\]
In this case $D_S$ is the following digraph, with adjacency matrix $A_S$:
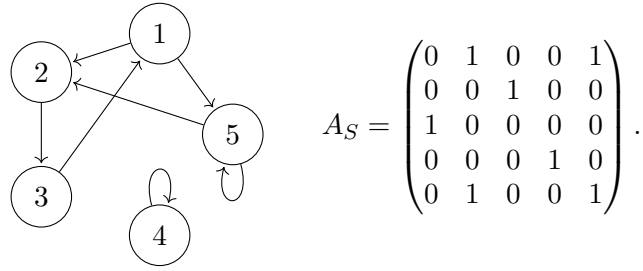
\[
\begin{tikzpicture}[baseline=(current bounding box.center),scale=.7,shorten >=1pt,->,auto,node distance=2cm,
                    main node/.style={circle,draw,minimum size=0.8cm}]
  \foreach \x in {1,...,5}
    \node[main node] (n\x) at (\x*72:2cm) {\x};
  
  \path
    (n1) edge (n2)
    (n1) edge (n5)
    (n2) edge (n3)
    (n3) edge (n1)
    (n4) edge [loop above] (n4)
    (n5) edge [loop below] (n5)
    (n5) edge (n2);
\end{tikzpicture}
\qquad
A_S = \begin{pmatrix}
0 & 1 & 0 & 0 & 1 \\
0 & 0 & 1 & 0 & 0 \\
1 & 0 & 0 & 0 & 0 \\
0 & 0 & 0 & 1 & 0 \\
0 & 1 & 0 & 0 & 1 \\
\end{pmatrix}.
\]
\end{example}
\begin{definition}
Two supports $S$ and $\tilde{S}$ are \emph{permutation similar} if there exists a permutation matrix $P$ such that $A_{\tilde{S}}=P A_SP^{-1}$.
\end{definition}
Since conjugation of the adjacency matrix by a permutation matrix corresponds to the natural action of the symmetric group by permuting vertex labels of the digraph, it follows that the digraphs $D_S$ and $D_{\tilde{S}}$ are isomorphic if and only if $S$ is permutation similar to $\tilde{S}$. 
\begin{proposition}
If $S$ and $\tilde{S}$ are permutation similar, then $N(\CC;S)=N(\CC;\tilde{S})$.  
\end{proposition}
\begin{proof}
Follows from the fact that $\CC=P\CC P^{-1}$ for each conjugacy class $\CC$ of $\MM_n(\Fq)$ and every $n\times n$ permutation matrix $P$.
\end{proof}

\begin{problem}\label{prob:prescribed}
  Given $\CC$ and $S$, determine $N(\CC;S)$.
\end{problem}
 Though the general case of the above problem is open we now discuss some cases where an answer can be gleaned from known results.

  \subsection{Hessenberg spaces and adjoint orbits}\label{subsec:hessenberg-spaces}
One can give a symmetric function formula for $N(\CC;S)$ for a specific class of supports $S$. We begin with some definitions. A generalization of the chromatic polynomial of a graph, called the chromatic symmetric function was originally introduced by Stanley~\cite{MR1317387}. A one parameter deformation of this function called the chromatic quasisymmetric function was defined by Shareshian and Wachs~\cite{MR3488041}. Let $G$ be a graph with vertex set $[n]$. Let $\PP$ denote the set of positive integers and consider functions $\kappa:G\to \PP$ defined on the vertex set of $G$. Denote by ${\rm asc}(\kappa)$, the number of edges $\{i,j\}$ of of $G$ with $i<j$ and $\kappa(i)<\kappa(j)$. The chromatic quasisymmetric function of $G$ is defined by
\begin{align*}
  X_G(x;t):=\sum_{\substack{\kappa:[n]\to \PP\\ \kappa \text{ proper}}}t^{{\rm asc}(\kappa)}x_{\kappa(1)}x_{\kappa(2)}\cdots x_{\kappa(n)},
\end{align*}
where the sum is taken over all \emph{proper} colorings (adjacent vertices do not receive the same color) of the vertex set of $G$.
\begin{definition}\label{def:hessenberg-function}
A Hessenberg function is a weakly increasing function $\mm:[n]\to [n]$ satisfying $\mm(i)\geq i$ for each $i\in [n]$.  
\end{definition}

It will be convenient to represent a Hessenberg function $\mm$ defined on $[n]$ by the tuple $(\mm(1),\ldots,\mm(n))$. To each such Hessenberg function $\mm$, we associate an indifference graph $G(\mm)$ which has vertex set $[n]$ and edge set $\{\{i,j\}:1\leq i<j\leq \mm(i)\}$. For this specific class of graphs it is known that the chromatic quasisymmetric function $X_{G(\mm)}(x;t)$ is in fact symmetric. 
\begin{definition}\label{def:hessenberg-variety}
Given an operator $T:\Fq^n\to \Fq^n$, the Hessenberg variety $\he(\mm,T)$ is defined as the collection of all complete flags $V_1\subset V_2\subset \cdots \subset V_n=\Fq^n$ such that $TV_i\subseteq V_{\mm(i)}$ for each $i\in [n]$.  
\end{definition}

The following result was recently proved by Abreu, Nigro and Ram~\cite[Thm.~1.1]{abreu2024}.
\begin{theorem}\label{thm:points}
 For each operator $T$ on $\Fq^n$ with similarity invariants $\LL$, the number of $\Fq$-rational points on the Hessenberg variety $\he(\mm,T)$ is given by 
  $$
  |\he(\mm,T)|=\langle F_\LL(x;q), \omega X_{G(\mm)}(x;q)\rangle.
  $$
\end{theorem}

To each Hessenberg function $\mm$ defined on $[n]$, we associate the Hessenberg space $\mathcal{H}_{\mm}=\mathcal{H}_{\mm}(q)$ which consists of all matrices in $A=(a_{ij})\in\MM_n(\Fq)$ satisfying $a_{ij}=0$ whenever $i>\mm(j)$. 
\begin{example}
If $\mm=(2,2,3,4)$, then $\mathcal{H}_\mm(q)$ consists of all matrices over $\Fq$ of the form
\begin{align*}
  \begin{pmatrix}
    * & * & * & *\\
    * & * & * & *\\
    0 & 0 & * & *\\
    0 & 0 & 0 & *
  \end{pmatrix}.
\end{align*}
The associated unit interval graph $G(\mm)$ has edge set $\{\{1,2\}\}$:
\begin{center}
\begin{tikzpicture}[baseline=-0.5ex, every node/.style={circle, draw, inner sep=1.5pt, minimum size=16pt}]
  \node (1) at (0,0) {$1$};
  \node (2) at (1.2,0) {$2$};
  \node (3) at (2.4,0) {$3$};
  \node (4) at (3.6,0) {$4$};
  \draw (1) -- (2);
\end{tikzpicture}
\end{center}
In this case $X_{G(\mm)}(x;t)=(1+t)e_{211}.$
\end{example}
The space $\mathcal{H}_\mm$ consists of matrices with support $S=\{(i,j):i\leq \mm(j)\}$. We refer to such a support as a Hessenberg support.

The following result is the matrix version of Theorem \ref{thm:points}.

\begin{theorem}\label{thm:hessmats}
  The number of $n\times n$ matrices in $\mathcal{H}_\mm(q)$ which have similarity invariants $\LL$ is given by
  \begin{align*}
  |\mathcal{H}_\mm(q)\cap \CC(\LL)|=  \frac{q^{n \choose 2}(q-1)^n}{c(\LL)} \langle F_\LL(x;q), \omega X_{G(\mm)}(x;q)\rangle,
  \end{align*}
  where $c(\LL)$ denotes the centralizer size corresponding to $\LL$.
\end{theorem}

Recall Definition \ref{def:class-type} of class type. For a fixed class type $\tau$ of size $n$ there always exists a matrix of type $\tau$ over $\Fq$ for sufficiently large prime powers $q$.

\begin{definition}\label{def:class-type-polynomiality-property}
Let $S\subseteq [n]\times [n]$, and let $V_S(q)\subseteq \MM_n(\Fq)$ denote the coordinate subspace of matrices supported on $S$. We say that $S$ has the \emph{class type polynomiality property} if, for every class type $\tau$ of size $n$, there exists a polynomial $P_{S,\tau}(t)$ such that, for every sufficiently large prime power $q$ and every conjugacy class $\CC\subseteq \MM_n(\Fq)$ of class type $\tau$,
\[
  |V_S(q)\cap \CC|=P_{S,\tau}(q).
\]
\end{definition}

Theorem \ref{thm:hessmats} implies that every Hessenberg support has the class type polynomiality property. Indeed, for a given class type $\tau$, the number of matrices in the Hessenberg space $\mathcal{H}_{\mm}(q)$ which lie in a conjugacy class of class type $\tau$ is a rational function in $q$ and, therefore, a polynomial in $q$ (since it takes integer values for sufficiently large prime powers $q$).

\begin{problem}\label{prob:class-type-polynomiality-supports}
  Determine all supports which have this class type polynomiality property. 
\end{problem}

\begin{corollary}
Let  $\mm$ be a Hessenberg function on $[n]$ and suppose $f=\prod_{i=1}^r f_i^{n_i}$ where the $f_i$ are distinct irreducible polynomials over $\Fq$ with $\deg f_i=d_i$ for $1\leq i\leq r$. The number of matrices in $\mathcal{H}_\mm(q)$ with characteristic polynomial $f$ is given by
  \begin{align*}
(q-1)^{n}q^{n\choose 2}\langle C_f(x;q), \omega X_{G(\mm)}(x;q) \rangle,
  \end{align*}
    where
    \begin{align*}
    C_f(x;t):=\prod_{i=1}^r p_{d_i}\circ h_{n_i}\left[ \frac{X}{t-1} \right].      
    \end{align*}
  \end{corollary}
  \begin{proof}
    Sum Theorem~\ref{thm:hessmats} over all conjugacy classes with characteristic polynomial $f$; the resulting symmetric function sum is Equation~\eqref{eq:cf}. 
  \end{proof}

\subsection{Ad-nilpotent ideals and adjoint orbits} 
Let $\mathfrak b_n(\Fq)$ denote the Lie algebra of $n\times n$ upper triangular matrices over $\Fq$ and let $\mathfrak u_n(\Fq)$ denote its nilradical, the Lie algebra of strictly upper triangular matrices. An \emph{ad-nilpotent ideal} of $\mathfrak u_n(\Fq)$ is an ideal of $\mathfrak u_n(\Fq)$ which is stable under the adjoint action of $\mathfrak b_n(\Fq)$. In type $A$, these ideals are parametrized by Hessenberg functions: if $\mm$ is a Hessenberg function on $[n]$, then
\[
\mathfrak u_\mm(\Fq):=\{(a_{ij})\in \mathfrak u_n(\Fq):a_{ij}=0\text{ whenever }j\leq \mm(i)\}
\]
is an ad-nilpotent ideal, and it is known that every ad-nilpotent ideal arises uniquely in this way. For instance, if $\mm=(2,3,4,4)$, then $\mathfrak u_\mm(\Fq)$ consists of all matrices over $\Fq$ of the form
$$
\begin{pmatrix}
0 & 0 & * & * \\
0 & 0 & 0 & * \\
0 & 0 & 0 & 0 \\
0 & 0 & 0 & 0
\end{pmatrix}.
$$
 The next result follows from Theorem 4.4 and Corollary 4.5 in Gagnon~\cite{MR4659580}. \begin{theorem}\label{thm:gag}
 The number of nilpotent matrices with Jordan form partition $\lambda$ that lie in the subspace $\mathfrak u_\mm(\Fq)$ is given by
$$
\frac{q^{{n \choose 2}-e} (q-1)^n}{c_q(1,\lambda)}\langle \widetilde{H}_\lambda(x;q),X_{G(\mm)}(x;q)\rangle,
$$
where $e$ denotes the number of edges of the graph $G(\mm)$.
\end{theorem}
Another proof of Theorem \ref{thm:gag} and a tableau formula for the count was given by Bardestani, Mallahi-Karai, Ram and Salmasian~\cite{BardestaniMallahiKaraiRamSalmasian2026}. 

\begin{remark}
 When $e=0$, the formula above agrees with Theorem \ref{thm:hessmats} for the number of nilpotent upper triangular matrices in the conjugacy class $\{(x,\lambda)\}$ since $\mm=(1,2,\ldots,n)$ in this case and  $X_{G(\mm)}=e_{1^n}=h_{1^n}$.  
\end{remark}

\section{Conclusion}\label{sec:conclusion}
Several classes of symmetric functions, such as Hall--Littlewood polynomials and their variants, $q$-Whittaker functions, and chromatic quasisymmetric functions, appear in the enumeration formulas for matrices with prescribed entries in an adjoint orbit. It would be interesting to identify further symmetric functions that occur in this setting.
 
The open problems raised above can be organized by theme:
\begin{itemize}
\item \emph{Affine slices and supports}: Problems~\ref{prob:main} and \ref{prob:prescribed};
\item \emph{Restrictions of partial linear maps}: Problems~\ref{prob:coeff} and \ref{prob:restriction-poset};
\item \emph{Symmetric-function positivity}: Conjecture~\ref{conj:skewHLschurpositive} and Problem~\ref{prob:skewHLschurcoeffs};
\item \emph{Polynomiality of counting functions}: Problem~\ref{prob:class-type-polynomiality-supports}.
\end{itemize}
With these themes in mind, the following questions seem natural.

\begin{question}\label{ques:symmetric-functions}
For which supports $S$ and prescriptions $\varphi$ can the affine-slice counts $N(\CC;S;\varphi)$ be expressed as scalar products of symmetric functions?
\end{question}

\begin{question}\label{ques:positive-combinatorics}
Do the extension counts $\eta((\mu,\MM),\LL)$ admit combinatorial formulas, for instance in terms of tableaux?
\end{question}

\section{Acknowledgements}
The author acknowledges with appreciation support from an Indo-Russian project, grant number DST/INT/RUS/RSF/P41/2021.

OpenAI's Prism LaTeX editor was used to proofread a manually prepared first draft of this article, and for expository rephrasing and numerical verification of combinatorial formulas. All mathematical statements and proofs are the sole responsibility of the author.

\printbibliography  
\end{document}